\documentclass[a4paper,11pt,twoside]{amsart}
\usepackage[english]{babel}
\usepackage[utf8]{inputenc}

\usepackage[a4paper,inner=3cm,outer=3cm,top=4cm,bottom=4cm,pdftex]{geometry}
\usepackage{fancyhdr}
\usepackage{color}
\usepackage{bold-extra}
\usepackage{ mathrsfs }
\usepackage{comment}
\usepackage{graphics}
\usepackage{aliascnt}
\usepackage[pdftex,citecolor=green,linkcolor=red]{hyperref}

\usepackage{amsmath}
\usepackage{amsfonts}
\usepackage{amssymb}
\usepackage{amsthm}
\usepackage{comment}
\usepackage{mathtools}
\usepackage{enumerate}
\usepackage{enumitem} 

\newtheorem{theorem}{Theorem}[section]

\newtheorem{lemma}[theorem]{Lemma}
\newtheorem{proposition}[theorem]{Proposition}
\theoremstyle{definition}

\newtheorem{remark}[theorem]{Remark}

\numberwithin{equation}{section}

\newcommand\eps{\varepsilon}

\newcommand\E{\mathbb{E}}

\newcommand\R{\mathbb{R}}
\newcommand\Z{\mathbb{Z}}

\newcommand\C{\mathbb{C}}

\newcommand{\CP}{\mathcal{P}}
\newcommand{\vdelta}{\underline{\delta}}
\newcommand{\wh}{\widehat}

\begin{document}
\title{Density versions of the binary Goldbach problem}

\author[Alsetri]{Ali Alsetri}
\address{Department of Mathematics, University of Kentucky\\
715 Patterson Office Tower\\
Lexington, KY 40506\\
USA}
\email{alialsetri@uky.edu}

\author[Shao]{Xuancheng Shao}
\address{Department of Mathematics, University of Kentucky\\
715 Patterson Office Tower\\
Lexington, KY 40506\\
USA}
\email{xuancheng.shao@uky.edu}
\thanks{XS was supported by NSF grant DMS-2200565.}

%\date{\today}

\maketitle
%\tableofcontents

\begin{abstract}
Let $\delta > 1/2$. We prove that if $A$ is a subset of the primes such that the relative density of $A$ in every reduced residue class is at least $\delta$, then almost all even integers can be written as the sum of two primes in $A$. The constant $1/2$ in the statement is best possible. Moreover we give an example to show that for any $\eps > 0$ there exists a subset of the primes with relative density at least $1 - \eps$ such that $A+A$ misses a positive proportion of even integers.
\end{abstract}

\section{Introduction}

 Let $\CP$ be the set of all primes and let $A\subset \CP$ be a subset. This paper studies the representation of even integers as sums of two primes belonging to $A$. The famous Goldbach conjecture, which remains wide open, states that every even integer $n \geq 4$ can be written as the sum of two primes. The ternary version which concerns representing odd integers as sums of three primes has been much more tractable. Vinogradov \cite{Vinogradov} proved in 1937 that every sufficiently large odd integer is a sum of three primes (see also \cite[Chapter 26]{Davenport}). This is now known to hold for all odd integers at least $7$, thanks to work of Helfgott \cite{Helfgott}.
 
Returning to the binary Goldbach problem, Estermann \cite{Estermann} showed in 1938 that almost every even integer can be written as the sum of two primes. More precisely if $E(N)$ denotes the set of even integers $n \leq N$ which cannot be written as the sum of two primes, then 
$$\frac{E(N)}{N} \ll_A (\log N)^{-A} $$
for every $A > 0$. A power saving for the error term was first obtained by Montgomery and Vaughan \cite{MontgomeryVaughan}, who showed that
$$\frac{E(N)}{N} \ll N^{-\delta} $$
for some positive constant $\delta > 0$. Since then there have been a series of improvements on the precise value of $\delta$, leading to the current record of $\delta = 0.28$ due to Pintz \cite{Pintz}.

In this paper we study Goldbach-type problems with primes restricted to subsets of $\CP$.
For a subset $A \subset \CP$, the relative lower density of $A$ in $\CP$ is defined by
$$ \vdelta(A) = \liminf_{N\rightarrow\infty} \frac{|A \cap [1,N]|}{|\CP \cap [1,N]|}. $$ 
In recent years density versions of Vinogradov's three primes theorem have been obtained \cite{LiPan, Shao, Shen}. For example, in \cite{Shao} it was proved that if $\vdelta(A) > 5/8$ then all sufficiently large odd positive integers can be written as a sum of three primes in $A$. See also \cite{MatomakiShao, MatomakiMaynardShao, Teravainen, Grimmelt} for results with (special) sparse subsets of primes.

Motivated by these results, we seek to obtain a density version of the almost all binary Goldbach problem. The binary problem for small positive density subsets of primes has been studied in \cite{RamareRuzsa, ChipeniukHamel, Matomaki}. In particular, Matom\"{a}ki \cite{Matomaki} proved that if $\vdelta(A) = \alpha$ for some positive constant $\alpha > 0$, then the sumset $A+A := \{p_1+p_2: p_1,p_2 \in A\}$ has positive lower density in the integers. Moreover, the lower density of $A+A$ is at least
$$ (e^{-\gamma}-o(1)) \frac{\alpha}{\log\log (1/\alpha)}, $$
where $\gamma$ is the Euler-Mascheroni constant and $o(1)$ denotes a quantity that tends to $0$ as $\alpha\rightarrow 0$. See also \cite{GrimmeltTeravainen} for related results with $A$ the set of almost twin primes.

 We seek conditions on $A \subset \CP$ which guarantee $A+A$ contains almost all even integers, or equivalently, $A+A$ has density $1/2$ in the integers. Specifically we ask whether there exists a positive constant $\alpha < 1$, such that if $\vdelta(A) \geq \alpha$ then $A+A$ contains almost all even integers. We show that, unlike the ternary case, such an $\alpha$ does not exist.
 
 \begin{theorem}\label{thm-binary}
For any $\eps > 0$ there exists a subset $A \subset \CP$ with $\vdelta(A) > 1-\eps$, such that a positive proportion of the even positive integers cannot be written as a sum of two primes in $A$.
\end{theorem}

However the situation changes if we impose additional local assumptions about the set $A$. For a reduced residue class $b\pmod W$, we define the relative lower density of $A$ in primes within this residue class by
$$ \vdelta(A; W,b) = \liminf_{N\rightarrow\infty} \frac{|A \cap \{1 \leq n\leq N: n \equiv  b\pmod{W}\}|}{|\CP \cap \{1 \leq n\leq N: n \equiv  b\pmod{W}\}|}. $$

\begin{theorem}\label{density-prime-1/2}
Let $A \subset \CP$ be a subset such that  
$$ \inf_{W,b}\vdelta(A; W,b) > 1/2, $$ 
where the infimum is taken over all reduced residue classes $b\pmod{W}$.  Then almost all even positive integers $N$ can be written as $N = p_1+p_2$ with $p_1,p_2 \in A$.
\end{theorem}

\begin{remark}
The constant $1/2$ is sharp. For any $\alpha > 2$ we may define
$$ A = \{p \in \CP: p \in [1,N_1] \cup [\alpha N_1,N_2] \cup [\alpha N_2,N_3]\cdots\}, $$
where $N_1 < N_2 < N_3<\cdots$ is a rapidly increasing sequence. Then $\vdelta(A; W, b) = 1/\alpha$ for all reduced residue classes $b\pmod{W}$ and $A+A$ misses a positive proportion of even integers.
\end{remark}

This result is proved using a variant of the Fourier analytic transference principle from additive combinatorics. This technique originated from the work of Green \cite{Green} who developed it to establish Roth's theorem in primes. Variants of the transference principle have been developed suitable for different problems. See \cite{Prendiville} for a survey.  For a variant suitable for additive problems involving dense subsets of the primes, see \cite[Section 6]{Matomaki} or \cite{ChipeniukHamel}. For an almost-all version of the transference principle, see \cite{Wang}. 

This article is organized as follows. In Section \ref{sec2} we study the binary Goldbach problem in the local setting, leading to the proof of Theorem \ref{thm-binary}.
In Section \ref{sec3} we develop an almost-all variant of the transference principle. In Section \ref{sec4} we use this transference principle to prove Theorem \ref{density-prime-1/2}.

\section{Local results}\label{sec2}

 For a positive integer $m$, we write $\Z_m^*$ for the set of reduced residue classes modulo $m$. In this section we will prove Theroem \ref{thm-binary} by studying the binary problem in the local setting of a cyclic group. We will first prove the following Theorem \ref{local-3/4} which is an independent result and will not be needed in the subsequent proofs in our paper, however as we will show Theorem \ref{thm-binary} is essentially a consequence of the observation that Theorem \ref{local-3/4} is sharp.
\begin{theorem}\label{local-3/4}
Let $m$ be an odd squarefree positive integer and let $A,B \subset \Z_m^*$ be subsets. Assume that 
$$ |A| +|B| > \varphi(m)\left(2 - \prod_{p \mid m} \frac{p-2}{p-1} \right). $$ 
Then $A+B = \Z_m$.
\end{theorem} 

\begin{proof}
Let $n \in \Z_m$ be arbitrary. Let
$$ X = \{x \in \Z_m^*: n-x \in \Z_m^*\}. $$
Then $x \in X$ if and only if $x \neq 0, n\pmod{p}$ for every $p\mid m$, and hence
$$ |X| \geq \prod_{p \mid m} (p-2). $$
It follows that the number of $x \in X$ such that $x \in A$ and $n-x \in B$ is at least
$$ |X| - |\Z_m^* \setminus A| - |\Z_m^*\setminus B| \geq |A| + |B| - 2\varphi(m) + \prod_{p \mid m}(p-2) > 0. $$
Pick any such $x$. Then $n = x + (n-x) \in A+B$, as desired.
\end{proof}

The lower bound for $|A| + |B|$ is sharp. Let $m = p_1p_2\cdots p_s$, where $p_1,\cdots,p_s$ are distinct odd primes. Define 
$$ A = \bigcup_{i=1}^s \{a \in \Z_m^*: a \pmod{p_i} \in X_i\}, \ \ 
B = \bigcup_{i=1}^s \{a \in \Z_m^*: a \pmod{p_i} \in Y_i\}, $$
where $X_i = Y_i = \{1\}$ for $1 \leq i \leq s-1$, and
$$ X_s = \{1,2,\cdots,x\}, \ \ Y_s= \{1,2,\cdots,y\}. $$
for some $1 \leq x, y < p_s$. Then $1 \notin A+B$ if $x+y \leq p_s$ and 
$$ |A| = \varphi(m)\left(1 - \frac{p_s-1-x}{p_s-1}\prod_{1\leq i \leq s-1} \frac{p_i-2}{p_i-1}\right), \ \ 
|B| = \varphi(m)\left(1 - \frac{p_s-1-y}{p_s-1}\prod_{1\leq i \leq s-1} \frac{p_i-2}{p_i-1}\right).
$$
Hence if we choose $x, y$ such that $x+y=p_s$ then
$$ |A| +|B| = \varphi(m)\left(2 - \prod_{p \mid m} \frac{p-2}{p-1} \right). $$ 
Moreover, if we choose $x = y = (p_s-1)/2$, then $A=B$ and we obtain $A \subset \Z_m^*$ with 
$$ |A| = \varphi(m)\left(1 - \frac{1}{2}\prod_{1\leq i \leq s-1} \frac{p_i-2}{p_i-1}\right) $$
such that $A+A \neq \Z_m$. Since the infinite product $\prod_p \frac{p-2}{p-1}$ diverges to $0$ we can suppose, for any $\eps > 0$, that $|A| > \varphi(m)(1-\eps).$ Now let $A'$ be the set of all primes which are congruent to some $a \in A$. By Dirichlet's theorem on primes in arithmetic progressions, $\vdelta(A') > 1-\eps$. Since $A+A \neq \Z_m$ it follows that $A'+A'$ does not contain any of the even integers in some fixed residue class modulo $m$ and so we immediately arrive at Theorem \ref{thm-binary}.

\section{A transference principle}\label{sec3}

We work in a cyclic group $\Z_N$. We adopt the normalization corresponding to the probability measure on the physical side $\Z_N$ and to the counting measure on the frequency side $\wh{\Z_N}$. Thus, the Fourier transform of a function $f: \Z_N\rightarrow \C$ is defined by
$$ \wh{f}(r) = \E_{n \in \Z_N} f(n) e_N(-rn) $$
for $r \in \Z_N$. For $p,q>0$, the norms $\|\wh{f}\|_p$ and $\|f\|_q$ are normalized as follows:
$$ \|\wh{f}\|_p = \Big(\sum_{r \in \Z_N} |\wh{f}(r)|^p\Big)^{1/p}, \ \ \|f\|_q = \Big(\E_{n \in \Z_N} |f(n)|^q\Big)^{1/q}. $$
For two functions $f_1,f_2:\Z_N\rightarrow\C$, their convolution $f_1*f_2$ is defined by
$$ f_1*f_2(n) = \E_{n_1 \in \Z_N} f_1(n_1) f_2(n-n_1). $$

\begin{proposition}\label{transference}
For $i \in \{1,2\}$, let $f_i, \nu_i: \Z_N\rightarrow \R_{\geq 0}$ be functions such that $f_i(n) \leq \nu_i(n)$ for every $n \in \Z_N$.  Let $\delta_i = \E_{n \in \Z_N} f_i(n)$. Let $\delta, \eta > 0$. Suppose that the following conditions hold.

\begin{enumerate}
\item $\delta_1+\delta_2 \geq 1+\delta $ for some $\delta > 0$. 

\item Each $f_i$ satisfies a mean value estimate in the sense that $\|\wh{f_i}\|_p \leq M$ for some $p \in (2,4)$ and $M \geq 1$.

\item Each $\nu_i$ has Fourier decay in the sense that $\|\wh{\nu_i-1}\|_{\infty} \leq c(\delta,\eta,p,M)$ for some sufficiently small constant $c(\delta,\eta,p,M)> 0$.

\end{enumerate}
Then $f_1*f_2(n) \geq \delta^3/1000$ for all but at most $\eta N$ values of $n \in \Z_N$.
\end{proposition}

To prove Proposition \ref{transference}, first we construct in Lemma \ref{lem:transference1} decompositions $f_i = g_i + h_i$ for each $i \in \{1,2\}$, such that $g_i$ is essentially $1$-bounded and $h_i$ is Fourier uniform in the sense that $\|\widehat{h_i}\|_{\infty} = o(1)$.  Then we show in Lemma \ref{lem:transference2} that $g_1*g_2(n) \gg_{\delta} 1$ for all $n \in \Z_N$ using hypothesis (1) about the sizes of $\delta_1,\delta_2$. Finally we show in Lemma \ref{lem:transference3} that $f_1*f_2(n) \gg_{\delta} 1$ for almost all $n$, using a standard Fourier analytic argument.

We now turn to the details. Let $\eps > 0$ be a small constant to be chosen later in terms of $\delta,\eta,p,M$.

\begin{lemma}\label{lem:transference1}
Let the notations and assumptions be as above. For each $i \in \{1,2\}$, we may construct an approximant $g_i : \Z_N \rightarrow \R_{\geq 0}$ of $f_i$ with the following properties:
\begin{enumerate}
\item $\E_{n \in \Z_N} g_i(n) = \delta_i$.
\item $\|g_i\|_{\infty}  \leq 1 + \delta/10$.
\item $\|\wh{f_i} - \wh{g_i}\|_{\infty} \leq \eps$.
\item $\|\wh{g_i}\|_p \leq M$.
\end{enumerate}
\end{lemma}

The statement of the lemma is analogous to \cite[Lemma 4.2]{Shao} (where we caution that the functions and the Fourier transforms are normalized differently), and the proof follows the same arguments as in \cite[Proposition 5.1]{GreenTao}. For completeness, we include a full proof.

\begin{proof}
For convenience, we drop the dependence on $i$, writing $f = f_i$, $g = g_i$, $\nu = \nu_i$.
Define the large spectrum of $f$ to be 
$$
R = \{r \in \Z_N : |\wh{f}(r)| \geq \eps \}.
$$ 
From the mean value estimate $\|\widehat{f}\|_p \leq M$ it follows that
$$ \eps^p |R| \leq \sum_{r \in \Z_N} |\widehat{f}(r)|^p \leq M^p, $$
and hence $|R| \leq (M/\eps)^p$. Define the  Bohr set 
$$
B = \{x \in \Z_N : |e_N(xr) -1| \leq \eps \text{ for each }r \in R\}.
$$ 
By the pigeonhole principle (see \cite[Lemma 4.20]{TaoVu}), it follows that $|B| \geq (c\eps)^{|R|}N$ for some absolute constant $c > 0$ and thus $|B| \gg_{\eps,p,M} N$. Now define the approximant $g: \Z_N \rightarrow \R_{\geq 0}$ by
$$  g(n) = \E_{b_1,b_2 \in B} f(n+b_1-b_2). $$
We will verify that $g$ satisfies the four desired properties.

Property (1) follows trivially by the definition of $g$. To verify property (2), note that for each $n \in \Z_N$ we have
$$ g(n) \leq \E_{b_1,b_2 \in B} \nu(n+b_1-b_2) = \sum_{r \in \Z_N} \widehat{\nu}(r) e_N(-rn) \left|\E_{b \in B} e_N(rb)\right|^2. $$
By the Fourier decay property of $\nu$, we may replace $\widehat{\nu}(r)$ above by $1_{r=0}$ at the cost of an error at most $c = c(\delta,\eta,p,M)$ for some sufficiently small constant $c(\delta,\eta,p,M) > 0$. It follows that
$$ g(n) \leq 1 + O\Big( c\sum_{r \in \Z_N} \left|\E_{b \in B} e_N(rb)\right|^2\Big) = 1 + O\Big(\frac{cN}{|B|}\Big) = 1 + O_{\eps,p,M}(c). $$
Since $\eps$ is chosen in terms of $\delta,\eta,p,M$, we may ensure that $g(n) \leq  1 + \delta/10$ by choosing $c$ sufficiently small in terms of $\delta,\eta,p,M$.

To verify property (3), note that for each $r \in \Z_N$ we have
$$ |\widehat{f}(r) - \widehat{g}(r)| = |\widehat{f}(r)|\left(1 - |\E_{b \in B} e_N(rb)|^2\right) \leq 2|\widehat{f}(r)| \E_{b \in B} |1 - e_N(rb)|. $$
We divide into two cases according to whether $r \in R$ or not. If $r \in R$, then $|e_N(rb)-1| \leq \eps$ for each $b \in B$ by the definition of the Bohr set $B$, and hence 
$$  |\widehat{f}(r) - \widehat{g}(r)| \leq 2\eps |\widehat{f}(r)| \leq 4\eps, $$
using the trivial bound 
\begin{equation}\label{eq:trivial}
|\widehat{f}(r)| \leq \E_{n \in \Z_N} f(n) \leq \E_{n \in \Z_N} \nu(n) = \widehat{\nu}(0) \leq 2 \end{equation}
(say). If $r \notin R$, then $|\widehat{f}(r)| \leq\eps$ by the definition of the large spectrum $R$, and hence $|\widehat{f}(r) - \widehat{g}(r)| \leq 4\eps$. This verifies property (3) (after replacing $\eps$ in our argument by $\eps/4$).

Finally, property (4) follows easily from the fact 
\begin{equation}\label{eq:ghat-bound}
 |\widehat{g}(r)| = |\widehat{f}(r)| \cdot |\E_{b \in B} e_N(rb)|^2 \leq |\widehat{f}(r)| 
\end{equation}

and the mean value estimate for $f$.
\end{proof}

\begin{lemma}\label{lem:transference2}
Let the notations and assumptions be as above, and let $g_1,g_2$ be the approximants constructed in Lemma \ref{lem:transference1}. Then $g_1*g_2(n) \geq \delta^3/200$ for every $n \in \Z_N$.
\end{lemma}

\begin{proof}
For $i \in \{1,2\}$, define $A_i$ to be the essential support of $g_i$:
$$ A_i = \{n \in \Z_N : g_i(n) \geq \delta/10\}. $$
Since $\|g_i\|_{\infty} \leq 1 + \delta/10$, we have
$$ \delta_iN = \sum_{n \in \Z_N} g_i(n) \leq \tfrac{1}{10}\delta N + \left(1+\tfrac{\delta}{10}\right) |A_i| \leq \tfrac{1}{5}\delta N + |A_i|. $$
Hence $|A_i| \geq (\delta_i - \delta/5) N$. Thus for every $n \in \Z_N$,
$$ |A_1 \cap (n-A_2)| \geq |A_1| + |A_2| - N \geq \left(\delta_1 + \delta_2 - 1 - \tfrac{2}{5}\delta\right) N \geq \tfrac{1}{2} \delta N, $$
and hence
$$ g_1*g_2(n) \geq N^{-1} \left(\tfrac{\delta}{10}\right)^2 |A_1 \cap (n-A_2)| \geq \tfrac{1}{200}\delta^3. $$
This completes the proof.
\end{proof}

\begin{lemma}\label{lem:transference3}
Let the notations and assumptions be as above, and let $g_1,g_2$ be the approximants constructed in Lemma \ref{lem:transference1}. Let $E$ be the exceptional set defined as
$$ E = \{n \in \Z_N: f_1*f_2(n) \leq \delta^3/1000\}, $$
Then $|E| \leq \eta N$.
\end{lemma}

\begin{proof}
Let $\alpha = |E|/N$. Consider the inner product $I = \langle g_1*g_2-f_1*f_2, 1_E\rangle$. On the one hand, we have
$$ I = \frac{1}{N} \sum_{n \in E} \left(g_1*g_2(n) - f_1*f_2(n)\right) \gg \delta^3 \alpha $$
by Lemma \ref{lem:transference2} and the definition of $E$.
On the other hand, by Plancherel's identity and Cauchy-Schwarz inequality we have
$$ I = \langle \wh{g_1}\wh{g_2} - \wh{f_1}\wh{f_2}, \wh{1_E}\rangle \leq \|\wh{1_E}\|_2\cdot \|\wh{g_1}\wh{g_2}-\wh{f_1}\wh{f_2}\|_2. $$
Clearly $\|\wh{1_E}\|_2 = \|1_E\|_2 = \alpha^{1/2}$. Thus the two inequalities above together imply that
$$ \alpha \ll \delta^{-6} \|\wh{g_1}\wh{g_2}-\wh{f_1}\wh{f_2}\|_2^2. $$
Note that for every $r \in \Z_N$ we have
$$ |\wh{g_1}(r)\wh{g_2}(r) - \wh{f_1}(r)\wh{f_2}(r)| \leq |\wh{g_1}(r)| \cdot |\wh{g_2}(r)-\wh{f_2}(r)| + |\wh{f_2}(r)| \cdot |\wh{g_1}(r) - \wh{f_1}(r)| \leq 4\eps, $$
where we used the property that $\|\wh{f_i}-\wh{g_i}\|_{\infty} \leq \eps$ and the trivial  bounds $|\widehat{f_i}(r)| \leq 2$ and $|\widehat{g_i}(r)|  \leq 2$ (see \eqref{eq:trivial} and \eqref{eq:ghat-bound}).
Moreover by Cauchy-Schwarz, condition $(4)$ of Lemma \ref{lem:transference1} and the mean value estimate $(2)$ of Proposition \ref{transference}, we have
$$
\|\wh{g_1}\wh{g_2}-\wh{f_1}\wh{f_2}\|_{p/2} \leq \|\wh{g_1}\wh{g_2}\|_{p/2} + \|\wh{f_1}\wh{f_2}\|_{p/2} \leq \|\wh{g_1}\|_p \|\wh{g_2}\|_p + \|\wh{f_1}\|_p\|\wh{f_2}\|_p \leq 2M^2. 
$$
Since $p \in (2,4)$, it follows that
$$ \|\wh{g_1}\wh{g_2}-\wh{f_1}\wh{f_2}\|_2 \leq \|\wh{g_1}\wh{g_2}-\wh{f_1}\wh{f_2}\|_{\infty}^{1-p/4} \cdot \|\wh{g_1}\wh{g_2}-\wh{f_1}\wh{f_2}\|_{p/2}^{p/4} \ll_{p,M} \eps^{1-p/4} $$
and hence
$$ \alpha \ll_{p,M} \delta^{-6} \eps^{2-p/2}. $$
Thus we may ensure that $\alpha \leq \eta$ by choosing $\eps = (c\delta^6\eta)^{2/(4-p)}$, where $c = c(p,M) > 0$ is a sufficiently small constant.
\end{proof}

As mentioned previously, the proof of Proposition \ref{transference} is completed by combining Lemmas \ref{lem:transference1}, \ref{lem:transference2} and \ref{lem:transference3}.

\section{Binary Goldbach for dense subsets of primes}\label{sec4}

In this section we prove Theorem \ref{density-prime-1/2}.  Let $A \subset \CP$ be a subset satisfying the assumptions in Theorem \ref{density-prime-1/2}. Choose $\delta \in (0,1/2)$ such that $\vdelta(A; W,b) \geq 1/2 + \delta$ for every reduced residue class $b\pmod{W}$. Let $E$ be the exceptional set consisting of those even positive integers that cannot be written as a sum of two primes in $A$. It suffices to  show that
$$ \lim_{M\rightarrow\infty} \frac{|E \cap [(1-\delta^{10})M, M]|}{M} = 0. $$
Suppose, for the purpose of contradiction, that there exists a constant $\eta > 0$ such that 
\begin{equation}\label{eq:E-lower-bound} 
|E \cap [(1-\delta^{10})M, M]| \geq \eta M
\end{equation}
for each $M = M_i$ in an infinite increasing sequence $M_1 < M_2 < \cdots$ of positive integers. Set $W = \prod_{p \leq z}p$, where $z$ is a constant sufficiently large in terms of $\delta,\eta$.  Recall that 
$$ \vdelta(A; W, b) = \liminf_{M\rightarrow\infty} \frac{|A_{W,b} \cap [1,M]|}{|\CP_{W,b} \cap [1,M]|}, $$
where, for a set $A \subset \Z$ and a residue class $b\pmod W$, the notation $A_{W,b}$ is defined by
$$ A_{W,b} = \{n \in A: n \equiv b\pmod{W}\}. $$
Thus, for some large value $M = M_i$, we have
$$ \frac{|A_{W,b} \cap [1,M]|}{|\CP_{W,b} \cap [1,M]|} \geq \frac{1}{2} + \frac{\delta}{2} $$
for each $b \in \Z_W^*$. Fix this value of $M$ for the remainder of the proof, and let $N = \lfloor M/W\rfloor$. 

By the pigeonhole principle, there exists an even residue class $r\pmod{W}$ such that
\begin{equation}\label{eq:EWr-lower-bound} 
|E_{W,r} \cap [(1-\delta^{10})M,M]| \geq \frac{\eta M}{W} \geq \eta N. 
\end{equation}
By the Chinese remainder theorem we can choose $b_1,b_2 \in \Z_W^*$ with $r = b_1+b_2$. For $i \in \{1,2\}$, define $f_i, \nu_i : \Z_N \rightarrow \R_{\geq 0}$ (naturally identifying $\Z_N$ with $\{1,2,\cdots,N\}$) by
$$ \nu_i(n) =
\begin{cases}
             \frac{\varphi(W)}{W}\log(Wn + b_i) &\text{if}\ Wn + b_i\in \CP, \\
 
     0 &\text{otherwise.}
\end{cases} $$
and
$$ f_i(n) =
\begin{cases}
             \frac{\varphi(W)}{W}\log(Wn + b_i) &\text{if}\ Wn + b_i \in A \text{ and } Wn+b_i \leq (1-\delta^5)M, \\
 
     0 &\text{otherwise.}
\end{cases} $$
We will show that $f_i,\nu_i$ satisfy the assumptions of Proposition \ref{transference}. Note that $\frac{\nu_i}{N}$ is the same as $\lambda_{b_i,W,N}$ in the notation of Green's paper \cite{Green} which we will appeal to. Clearly $0 \leq f_i(n) \leq \nu_i(n)$ for every $n$. The mean value estimates $\|\widehat{f_i}\|_p = O(1)$ follows from \cite[Lemma 6.6]{Green} with exponent $p=3$ (say). The Fourier decay property of $\nu_i$ follows from \cite[Lemma 6.2]{Green}, once $z$ is chosen large enough in terms of $\delta, \eta$. Now note that the average of $f_i(n)$ is 
$$ \delta_i = \E_{n \in \Z_N}f_i(n) = \frac{\varphi(W)}{NW}\sum_{p \in A_{W,b_i}\cap [1,(1-\delta^5)M]}\log{p}. $$
By restricting the sum over $p$ above to $p \in A_{W,b_i} \cap [M(\log M)^{-10}, (1-\delta^5)M]$ and noting that
\begin{equation*}
\begin{split}
|A_{W,b_i} \cap [M(\log M)^{-10}, (1-\delta^5)M]| &\geq |A_{W,b_i} \cap [1,M]| - M(\log M)^{-10} - \frac{\delta^4M}{\varphi(W)\log M} \\
 &\geq \left(\frac{1}{2} + \frac{\delta}{3}\right) \frac{M}{\varphi(W)\log M}, 
\end{split}
\end{equation*}
we deduce that
$$ \delta_i \geq \frac{\varphi(W)}{NW} \left(\frac{1}{2} + \frac{\delta}{3}\right) \frac{M}{\varphi(W)\log M} (\log M - 10\log\log M) \geq \frac{1}{2} + \frac{\delta}{4}. $$
Hence, we may apply Proposition \ref{transference} to the functions $f_i,\nu_i$ (with $\delta,\eta$ replaced by $\delta/2,\eta/2$, respectively) to  conclude that
$$ f_1*f_2(n) \gg \delta^3 $$
for all but at most $\eta N/2$ values of $n \in \Z_N$.  In view of \eqref{eq:EWr-lower-bound}, there exists $m \in E_{W,r} \cap [(1-\delta^{10})M, M]$ such that $(m-b_1-b_2)/W$ (naturally viewed as an element in $\Z_N$) is in the support of $f_1*f_2$. By the definition of $f_1,f_2$, this implies that we can write
$$ \frac{m-b_1-b_2}{W} \equiv n_1+n_2 \pmod{N} $$
for some positive integers $n_1,n_2$ with $Wn_i+b_i \in A$ and $Wn_i + b_i \leq (1-\delta^5)M$. The congruence above can be rewritten as
$$ m \equiv (Wn_1+b_1) + (Wn_2+b_2) \pmod{WN}. $$
Since $m \in [(1-\delta^{10})M, M]$ and $(Wn_1+b_1) + (Wn_2+b_2)  \leq (2-2\delta^5)M$, the congruence above must be an equality in the integers:
$$ m = (Wn_1+b_1) + (Wn_2+b_2). $$
This implies that $m$ can be written as the sum of two primes in $A$, contradicting $m \in E$.

\subsection*{Acknowledgements}
We thank the anonymous referee for helpful comments and suggestions.

\bibliographystyle{plain}
\bibliography{biblio}

\begin{thebibliography}{10}

\bibitem{ChipeniukHamel}
K.~Chipeniuk and M.~Hamel.
\newblock On sums of sets of primes with positive relative density.
\newblock {\em J. Lond. Math. Soc. (2)}, 83(3):673--690, 2011.

\bibitem{Davenport}
H.~Davenport.
\newblock {\em Multiplicative number theory}, volume~74 of {\em Graduate Texts
  in Mathematics}.
\newblock Springer-Verlag, New York, third edition, 2000.
\newblock Revised and with a preface by Hugh L. Montgomery.

\bibitem{Estermann}
T.~Estermann.
\newblock On {G}oldbach's problem: {P}roof that almost all even positive
  integers are sums of two primes.
\newblock {\em Proc. London Math. Soc. (2)}, 44(4):307--314, 1938.

\bibitem{Green}
B.~Green.
\newblock Roth's theorem in the primes.
\newblock {\em Ann. of Math. (2)}, 161(3):1609--1636, 2005.

\bibitem{GreenTao}
B.~Green and T.~Tao.
\newblock Restriction theory of the {S}elberg sieve, with applications.
\newblock {\em J. Th\'{e}or. Nombres Bordeaux}, 18(1):147--182, 2006.

\bibitem{Grimmelt}
L.~Grimmelt.
\newblock Vinogradov's theorem with {F}ouvry-{I}waniec primes.
\newblock {\em Algebra Number Theory}, 16(7):1705--1776, 2022.

\bibitem{GrimmeltTeravainen}
L.~Grimmelt and J.~Ter\"{a}v\"{a}inen.
\newblock The exceptional set in {G}oldbach's problem with almost twin primes,
  2022.
\newblock ArXiv 2207.08805.

\bibitem{Helfgott}
A.~H. Helfgott.
\newblock The ternary {G}oldbach problem.
\newblock 2015.
\newblock ArXiv 0903.4503.

\bibitem{LiPan}
H.~Li and H.~Pan.
\newblock A density version of {V}inogradov's three primes theorem.
\newblock {\em Forum Math.}, 22(4):699--714, 2010.

\bibitem{Matomaki}
K.~Matom\"{a}ki.
\newblock Sums of positive density subsets of the primes.
\newblock {\em Acta Arith.}, 159(3):201--225, 2013.

\bibitem{MatomakiMaynardShao}
K.~Matom\"{a}ki, J.~Maynard, and X.~Shao.
\newblock Vinogradov's theorem with almost equal summands.
\newblock {\em Proc. Lond. Math. Soc. (3)}, 115(2):323--347, 2017.

\bibitem{MatomakiShao}
K.~Matom\"{a}ki and X.~Shao.
\newblock Vinogradov's three primes theorem with almost twin primes.
\newblock {\em Compos. Math.}, 153(6):1220--1256, 2017.

\bibitem{MontgomeryVaughan}
H.~L. Montgomery and R.~C. Vaughan.
\newblock The exceptional set in {G}oldbach's problem.
\newblock {\em Acta Arith.}, 27:353--370, 1975.

\bibitem{Pintz}
J.~Pintz.
\newblock A new explicit formula in the additive theory of primes with
  applications {II}. the exceptional set in {G}oldbach’s problem, 2018.
\newblock ArXiv 1804.09084.

\bibitem{Prendiville}
Sean Prendiville.
\newblock Four variants of the {F}ourier-analytic transference principle.
\newblock {\em Online J. Anal. Comb.}, (12):Paper No. 5, 25, 2017.

\bibitem{RamareRuzsa}
O.~Ramar\'{e} and I.~Z. Ruzsa.
\newblock Additive properties of dense subsets of sifted sequences.
\newblock {\em J. Th\'{e}or. Nombres Bordeaux}, 13(2):559--581, 2001.

\bibitem{Shao}
X.~Shao.
\newblock A density version of the {V}inogradov three primes theorem.
\newblock {\em Duke Math. J.}, 163(3):489--512, 2014.

\bibitem{Shen}
Q.~Shen.
\newblock The ternary {G}oldbach problem with primes in positive density sets.
\newblock {\em J. Number Theory}, 168:334--345, 2016.

\bibitem{TaoVu}
T.~Tao and V.~H. Vu.
\newblock {\em Additive combinatorics}, volume 105 of {\em Cambridge Studies in
  Advanced Mathematics}.
\newblock Cambridge University Press, Cambridge, paperback edition, 2010.

\bibitem{Teravainen}
J.~Ter\"{a}v\"{a}inen.
\newblock The {G}oldbach problem for primes that are sums of two squares plus
  one.
\newblock {\em Mathematika}, 64(1):20--70, 2018.

\bibitem{Vinogradov}
I.~M. Vinogradov.
\newblock The representation of an odd number as a sum of three primes.
\newblock {\em Dokl. Akad. Nauk. SSSR.}, 16:139--142, 1937.

\bibitem{Wang}
Mengdi Wang.
\newblock Waring--{G}oldbach problem in short intervals.
\newblock {\em Israel J. Math.}, 261(2):637--669, 2024.

\end{thebibliography}

\end{document}